\documentclass{amsart}

\usepackage{latexsym}
\usepackage{amsmath}
\usepackage{amssymb}

\usepackage{multicol}

\usepackage[cp1251]{inputenc}
\usepackage[english]{babel}

\theoremstyle{plain}

\newtheorem{theorem}{Theorem}[section]
\newtheorem{proposition}{Proposition}[section]
\newtheorem{lemma}{Lemma}[section]
\newtheorem{corollary}{Corollary}[section]

\theoremstyle{definition}
\newtheorem{definition}{Definition}[section]

\newcommand{\neva}{N}

\newcommand{\bneva}{\mathfrak{N}}
\newcommand{\aream}{A}
\newcommand{\si}{\sigma}

\newcommand{\diin}{\Theta}

\newcommand{\er}{\varepsilon}
\newcommand{\za}{\zeta}

\newcommand{\ph}{\varphi}
\newcommand{\cph}{C_\varphi}
\newcommand{\hol}{\mathcal{H}ol}
\newcommand{\Dbb}{\mathbb D}
\newcommand{\Tbb}{\mathbb T}

\newcommand{\spd}{S_d}

\newcommand{\bd}{B_d}

\newcommand{\knl}{k}

\newcommand{\Nbb}{\mathbb N}

\numberwithin{equation}{section}

\begin{document}

\date{}

\title[Composition operators on model spaces]{Compact composition operators on model spaces}

\author{Evgueni Doubtsov}
\address{St.~Petersburg Department
of Steklov Mathematical Institute, Fontanka 27, St.~Petersburg 191023, Russia}

\email{dubtsov@pdmi.ras.ru}

\thanks{This research was supported by the Russian Science Foundation (grant No.~23-11-00171), https://rscf.ru/project/23-11-00171/.}

\begin{abstract}
Let $\varphi: B_d\to\mathbb{D}$, $d\ge 1$, be a holomorphic function,
where $B_d$ denotes the open unit ball of $\mathbb{C}^d$ and $\mathbb{D}= B_1$.
Let $\Theta: \mathbb{D} \to \mathbb{D}$ be an inner function
and let $K^p_\Theta$ denote the corresponding model space.
For $p>1$, we characterize the compact composition operators 
$C_\varphi: K^p_\Theta \to H^p(B_d)$, where $H^p(B_d)$ denotes the Hardy space.
\end{abstract}


\maketitle

\section{Introduction}

Let $\bd$ denote the open unit ball of $\mathbb{C}^d$, $d\ge 1$,
and let $\spd$ denote the unit sphere.
Let $\si=\si_d$ denote the normalized Lebesgue measure on the sphere $\spd$.
We also use symbols $\Dbb$ and $\Tbb$ for the unit disc $B_1$ and the unit circle $\partial B_1$, respectively.

For $d\ge 1$, let $\hol(\bd)$ denote the space of holomorphic functions in $\bd$.
For $0<p<\infty$, the classical Hardy space $H^p=H^p(\bd)$ consists of those $f\in \hol(\bd)$ for which
\[
\|f\|_{H^p}^p = \sup_{0<r<1} \int_{\spd} |f(r\za)|^p\, d\si_d(\za) < \infty.
\]
As usual, we identify the Hardy space $H^p(\bd)$, $p>0$, and the space
$H^p(\spd)$ of the corresponding boundary values.

It is well known that the composition operator $\cph: f\mapsto f\circ\ph$ sends $H^p(\Dbb)$
into $H^p(\bd)$, $p>0$. Indeed, let $f\in H^p(\Dbb)$. Then $|f|^p \le h$ for an appropriate harmonic
function $h$ on $\Dbb$. So, $|f\circ \ph|^p \le h\circ\ph$, hence $f\circ\ph \in H^p(\bd)$,
as required.

Since $\cph$ maps $H^2(\Dbb)$ into $H^2(\bd)$, it is natural to ask for a characterization of those $\ph$ for which
$\cph: H^2(\Dbb)\to H^2(\bd)$ is a compact operator.
Two-sided estimates for the essential norm of the operator $\cph: H^2(\Dbb) \to H^2(\bd)$, $d\ge 1$,
were obtained by B.R.~Choe \cite{Ch92} in terms of the corresponding pull-back measure.
A more explicit approach based on the Nevanlinna counting function was proposed in \cite{ShJoel87}
for $d=1$; see also \cite{AD20} for an extension to the case $d\ge 1$.

\begin{definition}\label{d_inner}
A
holomorphic function $\diin:\Dbb \to \Dbb$ is called \textsl{inner}
if $|\diin(\za)|=1$ for $\si_1$-a.e. $\za\in\Tbb$.
\end{definition}

In the above definition,
$\diin(\za)$ stands, as usual, for
$\lim_{r\to 1-} \diin(r\za)$.
Recall that the corresponding limit is known to exist $\si_1$-a.e.
Also, by the above definition, unimodular constants are not inner functions.

Given an inner function $\diin$ on $\Dbb$, the classical
model space $K_\diin$ is defined as
\[
K_\diin = H^2(\Dbb)\ominus \diin H^2(\Dbb).
\]
In this paper, firstly, we characterize those $\ph$ for which
$\cph: K_\diin \to H^2(\bd)$ is a compact operator.
For $d=1$, such characterizations were earlier obtained in \cite{LM13}.
Secondly, in Theorem~\ref{t_comp_KTp1}, we solve the analogous problem for $\cph: K_\diin^p \to H^p(\bd)$, $p>1$,
where $K_\diin^p := H^p \cap \diin \overline{H^p}$.
Note that $K_\diin^2 = K_\diin$.

\subsection*{Organization of the paper}
Auxiliary results, including Cohn's inequality and Stanton's formula, are presented in Section~\ref{s_aux}.
Compact composition operators $\cph: K_\diin \to H^2(\bd)$ are characterized in Section~\ref{s_KT}.
Real interpolation of Banach spaces is used in Section~\ref{s_KTp} to prove
that the compactness of $\cph: K_\diin^p \to H^p(\bd)$ does not depend on $p$ for $1<p<\infty$.

\section{Auxiliary results}\label{s_aux}
\subsection{Littlewood--Paley identity and related results}\label{ss_HL}

Given an $f\in H^2(\Dbb)$, the
Littlewood--Paley identity states that
\begin{equation}\label{e_LP}
\|f\|^2_{H^2(\Dbb)} = |f(0)|^2 + 2\int_{\Dbb} |f^\prime(w)|^2 \log\frac{1}{|w|}\, d\aream(w),
\end{equation}
where $\aream$ denotes the normalized area measure on $\Dbb$.

\subsubsection{Cohn's inequality}
Let $\diin: \Dbb \to \Dbb$ be an inner function.
If $f\in K_\diin$, then the lower estimate in \eqref{e_LP} is improvable
in the sense of the following theorem.

\begin{theorem}[\cite{Cn86}]\label{t_cohn}
Let $\diin$ be an inner function.
There exists $p\in (0,1)$ such that 
\begin{equation}\label{e_Cohn}
\|f\|^2_{H^2(\Dbb)} \ge |f(0)|^2 + C_p \int_{\Dbb} \frac{|f^\prime(w)|^2}{(1-|\diin(w)|)^p} \log\frac{1}{|w|}\, d\aream(w)
\end{equation}
for all $f\in K_\diin$.
\end{theorem}

\subsubsection{Stanton's formula}
To study the composition operator generated by a holomorphic self-map $\phi$ of the unit disk,
J.~H.~Shapiro \cite{ShJoel87} used for $f\circ\phi$ an analog of \eqref{e_LP}.
This analog is based on the Nevanlinna counting function $\neva_\phi$ defined as
\[
\neva_{\phi}(w) = \sum_{z\in \Dbb:\, \phi(z)=w} \log\frac{1}{|z|}, \quad w\in\Dbb\setminus\{\phi(0)\},
\]
where each pre-image is counted according to its multiplicity.
The key technical result in Shapiro's argument is the following Stanton formula. 

\begin{theorem}[\cite{ShJoel87}]\label{t_Stn_disk}
Let $\phi: \Dbb\to\Dbb$ be a holomorphic function.
Then
\begin{equation}\label{e_Stntn}
\|f\circ \phi\|^2_{H^2(\Dbb)} = |f(\phi(0))|^2 + 2\int_{\Dbb} |f^\prime(w)|^2
\neva_{\phi}(w) \, d\aream(w).
\end{equation}
\end{theorem}

Given an $f\in \hol(\bd)$ and a point $\za\in\spd$,
the slice-function $f_\za\in\hol(\Dbb)$ is defined by $f_\za(\lambda) = f(\lambda\za)$, $\lambda\in\Dbb$.

\begin{corollary}\label{c_Stanton}
Let $\ph: \bd\to\Dbb$, $d\ge 1$, be a holomorphic function.
Then
\begin{equation}\label{e_Stntn_d}
\|f\circ \ph\|^2_{H^2(\bd)} = |f(\ph(0))|^2 + 2\int_{\Dbb} |f^\prime(w)|^2
\left(\int_{\spd} \neva_{\ph_\za}(w)\,  d\si_d(\za) \right)\, d\aream(w).
\end{equation}
\end{corollary}
\begin{proof}
Let $\za\in \spd$.
Applying Theorem~\ref{t_Stn_disk} with $\phi = \ph_\za$,
and integrating with respect to the normalized Lebesgue measure $\si_d$
on $\spd$, we obtain \eqref{e_Stntn_d}.
\end{proof}

\subsection{Subharmonicity inequality for the Nevanlinna counting function}

\begin{proposition}[{\cite[Section 4.6]{ShJoel87}}]\label{p_subharm}
Let $w\in\Dbb$ and $\phi: \Dbb\to\Dbb$ be a holomorphic function.
Suppose that $\Delta$ is a disk centered at $w$ and such that $\phi(0)\notin \Delta$.
Then
\begin{equation}\label{e_subharm_1}
\neva_{\phi}(w) \le \frac{1}{\aream(\Delta)}\int_{\Delta} \neva_\phi(z)\, d\aream(z).
\end{equation}
\end{proposition}

\begin{corollary}\label{c_subharm}
Let $w\in\Dbb$ and $\ph: \bd\to\Dbb$ be a holomorphic function.
Suppose that $\Delta$ is a disk centered at $w$ and such that $\ph(0)\notin \Delta$.
Then
\begin{equation}\label{e_subharm}
\int_{\spd} \neva_{\ph_\za}(w)\,d\si_d(\za) \le \frac{1}{\aream(\Delta)}\int_{\Delta} 
\left(\int_{\spd} \neva_{\ph_\za}(z)\,  d\si_d(\za) \right)\, d\aream(z).
\end{equation}
\end{corollary}
\begin{proof}
Let $\za\in \spd$.
To obtain \eqref{e_subharm},
we apply Proposition~\ref{p_subharm} with $\phi = \ph_\za$,
integrate with respect $\si_d$, and use Fubini's Theorem.
\end{proof}

\subsection{Reproducing kernels for $K_\diin$}

Recall that the reproducing kernel $\knl_w(z)$ for $K_\diin$ is given by
\[
\knl_w(z) = \frac{1-\diin(z)\overline{\diin}(w)}{1- z\overline{w}}, \quad
\|\knl_w\|^2 = \frac{1-|\diin(w)|^2}{1-|w|^2}.
\]

Let $D_\er(w) = \{z\in\Dbb: |z-w| < \er|1- z\overline{w}|\}$, that is,
let $D_\er(w)$ denote the pseudohyperbolic $\er$-disk centered at $w\in \Dbb$.

\begin{lemma}[{\cite[Lemma~1]{LM13}}]\label{l_LM}
Let $\{w_n\}\subset\Dbb$ be such that $|w_n|\to 1$ and
\begin{equation}\label{e_LM}
|\diin (w_n)| < a
\end{equation}
for some $a\in (0,1)$. Then
\begin{itemize}
\item[(i)] $\knl_{w_n}/\|\knl_{w_n}\| \overset{w^\ast}\longrightarrow 0$ as $n\to \infty$;
\item[(ii)] there exist $\er>0$, $C>0$, and $n_0\in\Nbb$ such that
\[
|\knl_{w_n}^\prime(z)| \ge \frac{C}{(1-|w_n|^2)^2}, \quad z\in D_\er(w_n),
\]
\end{itemize}
for all $n\ge n_0$.
\end{lemma}

\section{Compact composition operators on $K_\diin$}\label{s_KT}

\begin{theorem}\label{t_cph_tri}
Let $\ph: \bd\to\Dbb$, $d\ge 1$, be a holomorphic function, and
let $\diin: \Dbb \to \Dbb$ be an inner function.
Then the following properties are equivalent.
\begin{itemize}
  \item [(i)] One has
  \begin{equation}\label{e_nevaTo0}
\int_{\spd}\neva_{\ph_\za}(w) \frac{1-|\diin(w)|}{1-|w|}\, d\si_d(\za)\to 0 \quad \textrm{as}\ |w|\to 1-.
\end{equation}
  \item [(ii)] $\cph: K_\diin \to H^2(\bd)$ is a compact operator.
\end{itemize}
\end{theorem}
\begin{proof}
The principal arguments below are modelled after \cite{ShJoel87}.

(i) $\Rightarrow$ (ii)
For $n\in \Nbb$, let
\[
K_{\diin, n}(\Dbb) = \{ \textrm{$f\in K_\diin(\Dbb)$: $f$ has zero of order $n$ at the origin}\}.
\]
Let $P_n: K_{\diin}(\Dbb)\to K_{\diin, n}(\Dbb)$ denote the orthogonal projector.
We claim that
\begin{equation}\label{e_tozero}
\|\cph P_n\|_{K_{\diin}(\Dbb) \to H^2(\bd)}\to 0
\end{equation}
as $n\to \infty$. Then $\cph$ is compact, since it is approximable by the finite-rank operators
$\cph (I-P_n)$.

To verify \eqref{e_tozero}, fix $\er>0$ and $f\in  K_{\diin}(\Dbb)$, $\|f\|\le 1$.
Let $g_n = P_n f$, then $\|g_n\|\le 1$.

Firstly, let $p\in (0,1)$ be that provided by Theorem~\ref{t_cohn}.
By \eqref{e_nevaTo0},
\begin{equation}\label{e_neva_p_to0}
\int_{\spd}\neva^p_{\ph_\za}(w) \frac{(1-|\diin(w)|)^p}{(1-|w|)^p}\, d\si_d(\za)\to 0 \quad \textrm{as}\ |w|\to 1-.
\end{equation}
For $r\in (0,1)$, put
\[
\bneva_{p, r} = \sup_{r< |w| <1} \int_{\spd} \neva_{\ph_\za}(w)\frac{(1-|\diin(w)|)^p}{1-|w|}\, d\si_d(\za).
\]

By Littlewood's inequality \cite{L25},
\[
\neva_{\ph_\za}(w) \le \log\left|\frac{1-\overline{w}\ph(0)}{\ph(0) -w}\right|,\quad w\in\Dbb\setminus \{\ph(0)\}.
\]
Thus, for $\frac{1+|\ph(0)|}{2} < |w| <1$, we obtain $\neva_{\ph_\za}(w) \le C (1-|w|)$, hence,
\[
\neva^{1-p}_{\ph_\za}(w) \le C (1-|w|)^{1-p}, \quad \frac{1+|\ph(0)|}{2} < |w| <1.
\] 
Therefore,
\begin{equation}\label{e_neva_p_only}
\bneva_{p, r}\to 0 \quad \textrm{as}\ \ r\to 1-
\end{equation}
by \eqref{e_neva_p_to0}.
Now, using \eqref{e_neva_p_only} and applying \eqref{e_Cohn} to $g_n \in {K_{\diin}}$, $\|g_n\|\le 1$, choose $R$ so close to $1$
 that
 \begin{equation}\label{e_Rdf}
\int_{\Dbb \setminus R\Dbb} \frac{|g_n^\prime(w)|^2}{(1-|\diin(w)|)^p} \bneva_{p, R} (1-|w|)\, d\aream(w) < 
C \|g_n\| \bneva_{p, R}  < \er
 \end{equation}
 for all $n\in\Nbb$.

Secondly,
\[
\max_{|w|<R}|g^\prime_n(w)|\to 0
\]
as $n\to\infty$.
Thus,
\begin{equation}\label{e_RRDdf}
\int_{R\Dbb} {|g_n^\prime(w)|^2}\int_{\spd} \neva_{\ph_\za}(w)\, d\si_d(\za)\, d\aream(w) < \er
\end{equation}
for all sufficiently large $n$.

By \eqref{e_Stntn_d},
\[
\begin{split}
\frac12  \|\cph g_n\|_{H^2(\bd)}^2
&= \int_{\Dbb} |g_n^\prime(w)|^2  \int_{\spd} \neva_{\ph_\za}(w)\, d\si_d(\za)\, d\aream(w) + |g_n(\ph(0))|^2
\\
&=\int_{\Dbb\setminus R\Dbb} + \int_{R\Dbb} +\ |g_n(\ph(0))|^2.
\end{split}
\]
Observe that $|g_n(\ph(0))|^2  \to 0$ as $n\to \infty$.
Hence, combining \eqref{e_Rdf} and \eqref{e_RRDdf}, we conclude that
\[
\|\cph g_n\|_{H^2(\bd)}\to 0,
\]
as required.

(ii) $\Rightarrow$ (i)
Assume that (i) does not hold.
Then there exists a sequence $\{w_n\}_{n=1}^\infty \subset \Dbb$ such that $|w_n|\to 1$ and
\begin{equation}\label{e_neva_Not0}
\int_{\spd}\neva_{\ph_\za}(w_n) \frac{1-|\diin(w_n)|}{1-|w_n|}\, d\si_d(\za)\ge c> 0.
\end{equation}
Sequentially applying \eqref{e_Stntn_d}, Lemma~\ref{l_LM}(ii) and Corollary~\ref{c_subharm}, we obtain
\begin{equation}\label{e_rker_estim}
\begin{aligned}
\|C_\ph \knl_{w_n}(z)\|^2/\|\knl_{w_n}\|^2 
    &\ge \frac{1}{\|\knl_{w_n}\|^2} \int_{\Dbb} |\knl^\prime_{w_n}(z)|^2
\left(\int_{\spd} \neva_{\ph_\za}(z)\,  d\si_d(\za) \right)\, d\aream(z)\\
    &\ge \int_{D(w_n, \er)} \frac{C}{(1- |w_n|^2)^{3}}
\left(\int_{\spd} \neva_{\ph_\za}(z)\,  d\si_d(\za) \right)\, d\aream(z)\\
    &\ge \frac{C_\er}{1- |w_n|^2} \int_{\spd} \neva_{\ph_\za}(w_n)\,d\si_d(\za).
\end{aligned}
\end{equation}
Now, recall that $C_\ph$ is a compact operator by (ii), 
thus, $\|C_\ph \knl_{w_n}(z)\|/\|\knl_{w_n}\| \to 0$ by Lemma~\ref{l_LM}(i).
Therefore, \eqref{e_rker_estim} contradicts \eqref{e_neva_Not0}.
The proof of the theorem is finished.
\end{proof}

\section{Compact composition operators on $K_\diin^p$}\label{s_KTp}

For $0<p<\infty$ and an inner function $\diin$, let 
\[
K^p_\diin = K^p_\diin(\Dbb) \overset{\mathrm{def}}= H^p(\Dbb) \cap \diin \overline{H^p}(\Dbb).
\]
It is well known and easy to see that $K^2_\diin = K_\diin$.

By definition, an inner function $\diin: \Dbb \to \Dbb$ is called one-component if
the set $\{z\in\Dbb: |\diin(z)|< r\}$ is connected for some $r\in (0,1)$.
The present section is motivated by the following assertion.

\begin{proposition}[\cite{Ba09, LM13}]\label{p_BaLM}
Let $\phi: \Dbb\to\Dbb$ be a holomorphic function, $p>1$ and let $\diin$ be a one-component inner function.
Then $C_\phi: K^p_\diin \to H^p(\Dbb)$ is a compact operator if and only if
\[
\neva_{\phi}(w) \frac{1-|\diin(w)|}{1-|w|} \to 0 \quad \textrm{as}\ |w|\to 1-.
\]
\end{proposition}
\begin{proof}
As indicated in \cite[Section~4]{LM13}, the results of Baranov \cite{Ba09} on compact Carleson embeddings
of the model spaces $K^p_\diin$
guarantee that for a \textit{one-component} inner function $\diin$, the compactness of the composition
operator $\cph: K^p_\diin(\Dbb)\to H^p(\Dbb)$ does not depend on $p\in (1, \infty)$.
Finally, for $p=2$, it suffices to apply Theorem~1 from \cite{LM13} or Theorem~\ref{t_cph_tri}.
\end{proof}

In Theorem~\ref{t_comp_KTp1} below, we show that 
the direct analog of Proposition~\ref{p_BaLM} holds for an arbitrary inner function $\diin$.
The arguments are based on 
the real interpolation method for Banach spaces, 
so we first recall 
related basic facts.

Let $(A_0, A_1)$ be a compatible couple of Banach spaces.
Given $0< \theta < 1$ and $1\le q \le \infty$, the real interpolation method
provides $(A_0, A_1)_{\theta, q}$, an interpolation space between $A_0$ and $A_1$
(see, e.g., \cite[Chapter~3]{BL76} for details).

We need the following one-sided compactness theorem for the real interpolation method.

\begin{theorem}[\cite{Cw92}]\label{t_cmpIntrp_Cw}
Let $(A_0, A_1)$ and $(B_0, B_1)$ be compatible couples of Banach spaces.
Assume that $T: A_j \to B_j$, $j=0, 1$, is a bounded linear operator such that 
$T:A_0 \to B_0$ is compact. 
Then $T: (A_0, A_1)_{\theta, q}  \to (B_0, B_1)_{\theta, q}$ is a compact operator
for all admissible $\theta$ and $q$.
\end{theorem}

\begin{theorem}\label{t_comp_KTp1}
Let $p>1$ and let $\diin$ be an inner function.
Then $\cph: K^p_\diin \to H^p(\bd)$ is a compact operator if and only if
property \eqref{e_nevaTo0} holds.
\end{theorem}

\begin{proof}
By Theorem~\ref{t_cph_tri}, it suffices to show that the compactness of 
$\cph: K^p_\diin(\Dbb)\to H^p(\bd)$ does not depend on $p\in (1, \infty)$.
To prove this property, assume that $p_0\in (1, \infty)$ and $\cph: K^{p_0}_\diin(\Dbb)\to H^{p_0}(\bd)$ is a compact operator.
We are planning to prove that 
$\cph: K^{p}_\diin(\Dbb)\to H^{p}(\bd)$ is also compact for all $p \in (1, \infty) \setminus \{p_0\}$.
In fact, for definiteness and without loss of generality, we may assume that $p> p_0$.

Fix $p$ and $p_1$ such that $p_1 > p > p_0$.
Define $\theta\in (0,1)$ by the following identity:
\begin{equation}\label{e_theta}
\frac{1}{p} = \frac{1-\theta}{p_0} + \frac{\theta}{p_1}.
\end{equation}
On the one hand, \eqref{e_theta} guarantees that
$(K_\diin^{p_0}, K_\diin^{p_1})_{\theta, p} = K_\diin^{p}$ (see, e.g., \cite{KZ20}
for more general results even in the bidisk $\Dbb^2$).
On the other hand, it is known that $(H^{p_0}(\bd), H^{p_1}(\bd))_{\theta, p} = H^{p}(\bd)$,
since $p_1 > p_0 > 1$.
Indeed, application of the classical Riesz projection
reduces interpolation between $H^{p_0}(\bd)$ and $H^{p_1}(\bd)$ to that between $L^{p_0}$ and $L^{p_1}$.

Now, observe that $\cph: K^{p_1}_\diin(\Dbb)\to H^{p_1}(\bd)$ is a bounded operator.
Indeed, as indicated in the introduction, $C_\ph: H^q(\Dbb) \to H^q(\bd)$ is bounded for all $0<q<\infty$.
Thus, applying Theorem~\ref{t_cmpIntrp_Cw} for $q=p$ and the pairs $(K_\diin^{p_0}, K_\diin^{p_1})$
and $(H^{p_0}(\bd), H^{p_1}(\bd))$, we conclude that  
$\cph: K^{p}_\diin(\Dbb)\to H^{p}(\bd)$ is compact, as required.
\end{proof}

\subsection*{Acknowledgements} The author is grateful to  the anonymous referees for
helpful corrections and remarks.

\bibliographystyle{amsplain}
\providecommand{\MR}{\relax\ifhmode\unskip\space\fi MR }

\end{document}